\documentclass{amsart}

\newcommand{\Ra}{\Rightarrow}
\newcommand{\w}{\omega}
\newcommand{\Z}{\mathcal Z}
\newcommand{\II}{\mathbb I}
\newcommand{\IN}{\mathbb N}
\newcommand{\IR}{\mathbb R}
\newcommand{\IQ}{\mathbb Q}
\newcommand{\lin}{\mathsf{lin}}
\newcommand{\aff}{\mathsf{aff}\!}

\newcommand{\V}{\mathcal V}

\newcommand{\W}{\mathcal W}
\newcommand{\e}{\varepsilon}

\newcommand{\M}{\mathcal M}
\newcommand{\I}{\mathcal I}
\newcommand{\DD}{\mathcal D}
\newcommand{\LL}{\mathcal L}

\newtheorem{theorem}{Theorem}
\newtheorem{problem}[theorem]{Problem}
\newtheorem{corollary}[theorem]{Corollary}
\newtheorem{proposition}[theorem]{Proposition}
\newtheorem{lemma}[theorem]{Lemma}

\begin{document}

\title{Detecting $\sigma Z_n$-sets in topological groups and linear metric spaces}
\author{Taras Banakh}
\address{Ivan Franko National University of Lviv (Ukraine) and Jan Kochanowski University in Kielce (Poland)}
\email{t.o.banakh@gmail.com}

\begin{abstract} We prove that if an analytic subset $A$ of a linear metric space $X$ is not contained in a $\sigma Z_\w$-subset of $X$ then for every Polish convex set $K$ with dense affine hull in $X$ the sum $A+K$ is non-meager in $X$ and the sets $A+A+K$ and $A-A+K$ have non-empty interior in the completion $\bar X$ of $X$. This implies two results:
\begin{itemize}
\item an analytic subgroup $A$ of a linear metric space $X$ is a $\sigma Z_\w$-space  if $A$ is not Polish and $A$ contains a Polish convex set $K$ with dense affine hull in $X$;
\item a dense convex analytic subset $A$ of a linear metric space $X$ is a $\sigma Z_\w$-space  if $A$ contains no open Polish subspace and $A$ contains a Polish convex set $K$ with dense affine hull in $X$.
\end{itemize}
\end{abstract}

\maketitle

A topological space $X$ is {\em analytic} if it is a metrizable continuous image of a Polish space. A {\em Polish space} is a separable topological space homeomorphic to a complete metric space. It is well-known \cite[14.2]{Ke} that each Borel subset of a Polish space is analytic. By Lusin-Sierpinski Theorem \cite{21.6}[Ke], each analytic subset $A$ of a Polish space $X$ has the {\em Baire property}, i.e., $(A\setminus U)\cup(U\setminus A)$ is meager in $X$ for some open set $U\subset X$.

By the classical result of S.~Banach \cite{Banach}, each non-complete analytic topological group is meager, i.e., can be represented as the countable union of nowhere dense subsets. 
This result can be easily derived from the following known fact attributed to Piccard \cite{Pic} and Pettis \cite{Pet} (see [9.9]{Ke}).

\begin{theorem}[Piccard-Pettis]\label{t:pp} If two analytic subsets $A,B$ of a Polish group $X$ are non-meager in $X$, then the set $AB$ has non-empty interior and $AA^{-1}$ is a neighborhood of unit in $G$.
\end{theorem}

Meager subsets of a topological space $X$ form a $\sigma$-ideal $\M(X)=\sigma\Z_0(X)$ which is the largest ideal among $\sigma$-ideals $\sigma \Z_n(X)$ generated by $Z_n$-sets in $X$. A subset $A\subset X$ of a topological space $X$ is called a {\em $Z_n$-set} in $X$ if $A$ is closed in $X$ and the complement $X\setminus A$ is $n$-dense in $X$. A subset $B\subset X$ is called {\em $n$-dense} in $X$ if the set $C(\II^n,B)$ of maps $\II^n\to B$ is dense in the space $C(\II^n,X)$ of all continuous functions $f:\II^n\to X$ defined on the $n$-dimensional cube $\II^n=[0,1]^n$. The function space $C(\II^n,X)$ is endowed with the  compact-open topology. Observe that a subset $D\subset X$ is dense if and only if $D$ is $0$-dense in $X$ and each $n$-dense set $D\subset X$ is $m$-dense in $X$ for every $m\ge n$.

The following properties of $Z_n$-sets follow immediately from the definitions:
\begin{itemize}
\item a subset $A\subset X$ is a $Z_0$-set if and only if $A$ is closed and nowhere dense in $X$;
\item for any numbers $0\le n\le m\le\w$ every $Z_m$-set in $X$ is a $Z_n$-set in $X$;
\item a subset $A\subset X$ is a $Z_\w$-set in $X$ if and only if $A$ is a $Z_n$-set in $X$ for every $n\in\IN$.
\end{itemize}
By $\sigma \Z_n(X)$ we shall denote the $\sigma$-ideal generated by $Z_n$-sets in $X$. It consists of subsets that can be covered by countably many $Z_n$-sets in $X$.
A topological space $X$ is called a {\em $\sigma Z_n$-space} if $X\in\sigma \Z_n(X)$. It follows that $\sigma\Z_m(X)\subset\sigma\Z_n(X)$ for any numbers $0\le n\le m\le\w$. So, the $\sigma$-ideal $\sigma\Z_\w(X)$ is the smallest ideal among the $\sigma$-ideals $\sigma\Z_n(X)$.

$Z_\w$-Sets and $\sigma Z_\w$-spaces play an important role in Infinite-Dimensional Topology, see \cite{BRZ}, \cite{BP}, \cite{Chap}, \cite{vM}, \cite{vM2}. In \cite[4.4]{DM} Dobrowolski and Mogilski asked the following problem related to the mentioned classical result of Banach \cite{Banach}.

\begin{problem}[Dobrowolski, Mogilski, 1990]\label{pr1} Is each non-complete analytic linear metric space a $\sigma Z_\w$-space?
\end{problem}

This problem was answered in negative by Banakh \cite{Ban99} (see also \cite[5.5.19]{BRZ}) who proved that the linear hull $\lin(E)$ of the Erd\"os set $E=\ell_2\cap\IQ^\w$ in the separable Hilbert space $\ell_2$ fails to be a $\sigma Z_\w$-space.

Yet, the following weaker version of Problem~\ref{pr1} still remains open (see \cite{Ban97}, \cite[2.2]{BCZ}).

\begin{problem}[Banakh, 1997]\label{pr2} Is each non-complete analytic linear metric space a $\sigma Z_n$-space for every $n\in\IN$?
\end{problem}

In this paper we shall give some partial positive answers to Problems~\ref{pr1} and \ref{pr2}, detecting analytic subsets in metrizable topological groups $G$ that belong to the $\sigma$-ideals $\sigma \Z_n(G)$ for $n\le \w$. In fact, we shall work with the smaller $\sigma$-ideals $\sigma\dot\Z_{\mathcal D}(G)$ and $\sigma\dot Z_n(G)$ defined as follows.

By a {\em metrizable group} we shall understand a metrizable topological group. It is known that for any metrizable group $G$ there exists a completely-metrizable group $\bar G$ containing $G$ as a dense subgroup. The group $\bar G$ is unique up to isomorphism and is called {\em the Raikov completion} of $G$. The Raikov completion of a separable metrizable group is a Polish group. For two subsets $A,B$ of a group $G$ by $A\cdot B$ or just $AB$ we denote their product $\{ab:a\in A,\;b\in B\}$ in $G$.

Let $G$ be a topological group and $\bar G$ be its Raikov completion. Let $\mathcal D$ be a family of subsets of $G$. A closed subset $A\subset G$ is called {\em a $\dot Z_{\mathcal D}$-set} in $X$ if there exists set $D\in\mathcal D$ such that the set $D\cdot \bar A$ has empty interior in $\bar G$, where $\bar A$ denotes the closure of $A$ in $\bar G$. By $\sigma\dot\Z_{\mathcal D}(G)$ we denote the $\sigma$-ideal generated by $\dot Z_{\mathcal D}$-sets in $G$.

\begin{proposition}\label{p:D=>n} Let $\mathcal D$ be a family of $n$-dense subsets of a topological group $G$. Then each $\dot Z_{\mathcal D}$-set $A$ in $G$ is a $Z_n$-set in $G$ and hence $\sigma\dot Z_\DD(G)\subset \sigma\Z_n(G)$.
\end{proposition}

\begin{proof} Assume that $A$ is a $\dot Z_\DD$-set in $X$. Given a continuous map $f:\II^n\to G$ and a neighborhood $U_0\subset G$ of the unit $1_G$, we need to find a continuous map $f':\II^n\to G\setminus A$ such that $f'(z)\in f(z)\cdot U_0$ for all $z\in \II^n$. Let $\bar G$ be the Raikov completion of the topological group $G$ and $\bar A$ be the closure of $A$ in $\bar G$.

Find an open neighborhood $\tilde U_0\subset\bar G$ of the unit $1_G$ such that $\tilde U_0\cap G=U_0$ and choose a neighborhood $\tilde U_1\subset \bar X$ of $1_G$ such that $\tilde U_1\tilde U_1\tilde U_1\subset \tilde U_0$. Since $A$ is a $\dot Z_\DD$-set in $G$, there exists a set $D\in\DD$ such that the set $D\cdot \bar A$ has empty interior in $\bar G$. The $n$-density of the set $D$ in $G$ implies the $n$-density of its inverse $D^{-1}=\{x^{-1}:x\in D\}$. Then there exists a continuous map $f_1:\II^n\to D^{-1}$ such that $f_1(z)\in f(z)\cdot \tilde U_1$ for all $z\in\II^n$.

Since the set $D\cdot\bar A$ has empty interior in $\bar G$, there is a point $u\in \tilde U_1\setminus D\cdot \bar A$. For this point we get $(D^{-1}\cdot u)\cap \bar A=\emptyset$. Consider the map $f_2:\II^n\to \bar G$, $f_2:z\mapsto f_1(z)u$, and observe that $f_2(\II^n)\cap\bar A_n\subset (D^{-1}\cdot u)\cap\bar A_n=\emptyset$. Since the set $f_2(\II^n)$ is compact, there is a neighborhood $\tilde U_2\subset \tilde U_1$ of the unit $1_G$ such that $(f_2(\II^n)\cdot \tilde U_2)\cap\bar A=\emptyset$. Using the density of $G$ in $\bar G$, choose a point $w\in G\cap (\tilde U_2\cdot u)$. Then the map $f_3:\II^n\to\bar G$ defined by $f_3(z)=f_2(z)\cdot u^{-1}w=f_1(z)\cdot uu^{-1}w=f_1(z)\cdot w\in G$ for $z\in\II^n$ has the properties: $f_3(\II^n)\subset G\setminus\bar A=G\setminus A$ and for every $z\in\II^n$
$$f_3(z)=f_1(z)uu^{-1}w\in f_1(z)\tilde U_1\tilde U_2\subset f(z)\tilde U_1\tilde U_1\tilde U_2\subset f(z)\tilde U_0,$$which implies $f(z)^{-1}f_3(z)\in G\cap\tilde U_0=U_0$ and finally $f_3(z)\in f(z)U_0$. The map $f_3:\II^n\to G\setminus A$ witnesses that $A$ is a $Z_n$-set in $G$.
\end{proof}

For a topological group $G$ by $\DD_n(G)$ we shall denote the family of all $n$-dense subsets in $G$. To simplify notation, $\dot Z_{\DD_n(G)}$-sets will be called {\em $\dot Z_n$-sets} in $G$. Also we shall denote the $\sigma$-ideal $\sigma \dot\Z_{\DD_n(G)}(G)$ by $\sigma \dot Z_n(G)$. This $\sigma$-ideal is generated by all $\dot Z_n$-sets in $G$. It consists of subsets that can be covered by countably many $\dot Z_n$-sets in $G$. Proposition~\ref{p:D=>n} implies that
$$\sigma\dot \Z_n(G)\subset \sigma\Z_n(G)$$ for any topological group $G$.
$\dot Z_n$-Sets in separable metrizable groups admit the following convenient characterization.

\begin{proposition}\label{p:dotZn} A closed subset $A$ of a separable metrizable group $G$ is a $\dot Z_n$-set in $G$ for some $n\le \w$ if and only if there exists a $\sigma$-compact $n$-dense subset $D\subset G$ such that for every compact set $K\subset D$ the set $K\cdot D$ is nowhere dense in $G$.
\end{proposition}

\begin{proof} Since $G$ is separable and metrizable, the Raikov completion $\bar G$ of $G$ is a Polish group. To prove the ``if'' part, assume that there exists a $\sigma$-compact $n$-dense subset $D\subset G$ such that for every compact set $K\subset D$ the set $K\cdot A$ is nowhere dense in $G$. Then the set $K\cdot\bar A\subset \overline{K\cdot A}$ is nowhere dense in $\bar G$ and the set $D\cdot\bar A$ is meager in $\bar G$. Since $\bar G$ is Polish, the set $D\cdot\bar A$ has empty interior in $\bar G$ and hence $A$ is a $\dot Z_n$-set in $G$.

To prove the ``only if'' part, assume that $A$ is a $\dot Z_n$-set and find an $n$-dense subset $D'\subset G$ such that the set $D'\cdot\bar A$ has empty interior in $\bar G$. The the function space $C(\II^n,D')$ is dense in $C(\II^n,G)$. Since the function space $C(\II^n,D')$ is metrizable and separable, we can find a countable dense subset $\{f_k\}_{k\in\w}$ in $C(\II^n,D')$. Then $D=\bigcup_{k\in\w}f_k(\II^n)$ is a $\sigma$-compact $n$-dense subset in $G$. It remains to show that for each compact set $K\subset D$ the set $K\cdot \bar A$ is nowhere dense in $\bar G$. Consider the multiplication map $\mu:K\times\bar A\to \bar G$, $\mu:(x,y)\mapsto xy$, and observe that for any compact subset $C\subset \bar G$ the preimage $\mu^{-1}(C)=\{(x,y)\in K\times\bar A:xy\in C\}\subset K\times (K^{-1}C)$ is compact. By \cite[3.7.18]{En}, the map $\mu$ is closed, which implies that the set $K\bar A=\mu(K\times\bar A)$ is closed in $\bar G$. Since the set $D\times \bar A$ has empty interior in $\bar G$, the closed subset $K\bar A\subset D\bar A$ is nowhere dense in $\bar G$. Then its subset $KA$ is nowhere dense in $G$.
\end{proof}

Let $\DD$ be a family of subsets of a topological group $G$. A subset $T\subset G$ is called {\em $\DD$-thick} if for every non-empty open set $U\subset T$ there exist a set $D\in\DD$ and a countable set $C\subset G$ such that  $D\subset C\cdot \bar U$. A set $T\subset G$ is called {\em $n$-thick} in $G$ if it is $\DD_n(G)$-thick. The latter means that for every non-empty open set $U\subset T$  there is a countable set $C\subset G$ such that the set $C\bar U$ in $n$-dense in $G$.

\begin{theorem}\label{t1} Let $\DD$ be a family of subsets in a separable metrizable group $G$. If an analytic subset $A$ of $G$ does not belong to the $\sigma$-ideal $\sigma\dot Z_\DD(X)$, then for any $\DD$-thick subset $T\subset G$ and any dense Polish subspace $P\subset T$ the set $PA$ is not meager in $G$, the set $PAPA$ has non-empty interior in the Raikov completion $\bar G$ of $G$, and the set $PAA^{-1}P^{-1}$ is a neighborhood of unit in $\bar G$.
\end{theorem}

\begin{proof} Assume that $A\notin\sigma\dot\Z_\DD(G)$ and $T$ is an $\DD$-thick set in $G$.
 On the Polish group $\bar G$ consider the $\sigma$-ideal $\I$ generated by the family $\{\bar A:A\in\sigma\dot Z_\DD(G)\}$ of closed subsets of the Polish group $\bar G$. It follows from $A\notin\sigma\dot Z_\DD(G)$ that $A\notin\I$. By the Solecki dichotomy \cite{Sol}, the analytic set $A\notin\I$ contains a Polish subspace $B\notin\I$. Replacing $B$ by a smaller closed subset of $B$, we can assume that each non-empty open subspace $U\subset B$ does not belong to the ideal $\I$.

Given a dense Polish subspace $P\subset T$, we shall show that the set $PB$ is not meager in $G$. To derive a contradiction, assume that $PB$ is meager in $G$ and find closed nowhere dense subsets $N_k\subset \bar G$, $k\in\w$, such that $PB\subset \bigcup_{k\in\w}N_k$. By the continuity of the multiplication in $G$, for every $k\in\w$ the set $M_k=\{(x,y)\in P\times B:xy\in N_k\}$ is closed in the Polish space $P\times B$. Since $P\times B\subset\bigcup_{k\in\w}M_k$, we can apply the Baire Theorem and find two non-empty open sets $V\subset P$ and $U\subset B$ such that $V\times U\subset M_k$ for some $k\in\w$. It follows that the set $\bar V\times \bar U\subset N_k$ is nowhere dense in $\bar G$. Here $\bar V$ is the closure of $V$ in $G$ and $\bar U$ is the closures of $U$ in $\bar G$.

Since the set $T$ is $\DD$-thick in $G$, and the set $\bar V\cap T$ has non-empty interior in $T$, for some countable set $S\subset G$ the set $S\cdot \bar V$ contains a set $D\in\DD$. 

By the choice of $P$, the non-empty open set $U\subset P$ does not belong to the ideal $\I$ and hence  $\bar U\cap G$ is not a $\dot Z_\DD$-set in $G$. Then for the set $D\in\DD$ the set $D\bar U$ has non-empty interior in $\bar G$ and hence is not meager in $\bar G$. On the other hand, the set $D\bar U\subset S\bar V\bar U\subset S\cdot N_k$ is meager in $\bar G$ being the union of countably many translations of the nowhere dense set $N_k$.  This contradiction shows that the set $PB$ is not meager in $G$ and consequently the analytic set $PA\supset PB$ is not meager in the Polish group $\bar G$. By the Piccard-Pettis Theorem~\ref{t:pp}, the set $PAPA$ has non-empty interior in $\bar G$ and the set $PA(PA)^{-1}$ is a neighborhood of the unit in $\bar G$.
\end{proof}

A topological space $X$ is called {\em densely-Polish} if $A$ contains a dense Polish subspace. It is known that an analytic space $A$ is densely-Polish if and only if $A$ is Baire.

\begin{corollary}\label{c:main} Let $\DD$ be a family of subsets of a separable metrizable group $G$. If analytic subsets $A,B$ of $G$ do not belong to the ideal $\sigma\dot Z_\DD(G)$, then for any densely-Polish $\DD$-thick sets $E,F$ in $X$ the sets $EA$, $FB$ are not meager in $G$ and the sets $EAFB$ and $EAB^{-1}F^{-1}$ have non-empty interior in the Raikov completion $\bar G$ of $G$.
\end{corollary}

\begin{proof} Let $E_*\subset E$ and $F_*\subset F$ be dense Polish subspaces of the densely-Polish spaces $E$ and $F$, respectively. By Theorem~\ref{t1}, the analytic sets $E_*A$ and $F_*B$ are not meager in the Polish space $\bar G$. By the Piccard-Pettis Theorem~\ref{t:pp}, the sets $E_* AF_* B\subset EAFB$ and $E_*AB^{-1}F_*^{-1}\subset EAB^{-1}F^{-1}$ have non-empty interior in the Polish group $\bar G$.
\end{proof}

Corollary~\ref{c:main} implies the next three corollaries.

\begin{corollary}\label{c:subgroup-D} Let $\DD$ be a family of subsets in a separable metrizable group $G$ and $A$ be an analytic subgroup in $G$. If $A\notin\sigma\dot Z_\DD(X)$, then for any densely-Polish $\DD$-thick subsets $E,F\subset G$ the set $EAF^{-1}$ have non-empty interior in the completion $\bar G$ of $G$.
\end{corollary}

\begin{corollary}\label{c:non-complete-D} Let $\DD$ be a family of subsets of a separable metrizable group $G$. If $G$ is not Polish and $G$ contains a densely-Polish $\DD$-thick subset $P$, then each analytic subset $A$ of $X$ belongs to the $\sigma$-ideal $\sigma \dot Z_\DD(X)$.
\end{corollary}

\begin{proof} By Corollary~\ref{c:main}, for every analytic set $A\notin\sigma\dot Z_\DD(G)$ of $G$ the set $PAPA\subset G$ has non-empty interior in the Raikov completion $\bar G$ of $G$. Then $G$ also has non-empty interior in $\bar G$ and hence coincide with the Polish group $\bar G$, which is a desired contradiction.
\end{proof}

A subset $A$ of an abelian group $G$ is called {\em additive} if $A+A\subset A$. In particular, each subgroup of $G$ is an additive set. Corollary~\ref{c:main} implies:

\begin{corollary}\label{c:additive-D} Let $\DD$ be a family of subsets in an abelian separable metrizable group $G$ and $A$ be an additive set in $G$. If $A\notin\sigma\dot Z_\DD(X)$, then for any densely-Polish $\DD$-thick subsets $E,F\subset X$ the set $A+E+F$ has non-empty interior in the Raikov completion $\bar G$ of $G$.
\end{corollary}

A similar result holds for convex subsets in linear metric spaces.

\begin{corollary}\label{c:convex-D} Let $\DD$ be a family of subsets of a separable linear metric space $X$, and let $A$ be a convex subset of $X$. If $A\notin\sigma\dot Z_\DD(X)$, then for any densely-Polish $\DD$-thick subsets $E,F\subset X$ the set $A+E+F$ has non-empty interior in the completion $\bar X$ of $X$.
\end{corollary}

\begin{proof} It follows that the homothetic copy $\frac12A=\{\frac12a:a\in A\}$ of $A$ does not belong to the ideal $\sigma\dot Z_\DD(X)$. By Corollary~\ref{c:main}, the set $\frac12A+\frac12A+E+F$ has non-empty interior in $\bar X$. The convexity of $A$ guarantees that $\frac12A+\frac12A\subset A$ and hence the set $A+E+F\supset\frac12A+\frac12A+E+F$ has non-empty interior in $\bar X$, too.
\end{proof}

Applying the above results to the family $\DD_n(G)$ of $n$-dense subsets in a topological group $G$, we get the following corollaries. In these corollaries we use the obvious fact that a topological group $G$ containing an $n$-thick separable subset is separable. By Proposition~\ref{p:D=>n}, $\sigma\dot\Z_n(G):=\sigma\dot\Z_{\DD_n(G)}(G)\subset\sigma\Z_n(G)$. By Proposition~\ref{p:dotZn}, a closed subset $A$ of a separable metrizable group $G$ is a $\dot Z_n$-set in $X$ if and only if there exists a $\sigma$-compact $n$-dense set $D\subset G$ such that for every compact set $K\subset D$ the set $K\cdot A$ is nowhere dense in $G$. 

We recall that a subset $T$ of a topological group $G$ is $n$-thick if and only if for any non-empty open set $U\subset T$ there is a countable subset $A\subset G$ such that the set $A\cdot U$ is $n$-densein $G$. Observe that each non-empty subset of a separable metrizable group is $0$-thick. Because of that the following corollary of Theorem~\ref{t1} can be considered as a generalization of the Piccard-Pettis Theorem~\ref{t:pp}.

\begin{corollary} If for some $n\le\w$ an analytic subset $A$ of a metrizable group $G$ does not belong to the $\sigma$-ideal $\sigma\dot Z_n(X)$, then for any $n$-thick subset $T\subset G$ and any dense Polish subspace $P\subset T$ the set $PA$ is not meager in $G$, the set $PAPA$ has non-empty interior in $\bar G$, and the set $PAA^{-1}P^{-1}$ is a neighborhood of unit in $\bar G$.
\end{corollary}

\begin{corollary}\label{c:main2} If for some $n\le\w$ analytic subsets $A,B$ of a metrizable group $G$ do not belong to the ideal $\sigma\dot Z_n(G)$, then for any densely-Polish $n$-thick sets $E,F$ in $X$ the sets $EA$, $FB$ are not meager in $G$ and the sets $EAFB$ and $EAB^{-1}F^{-1}$ have non-empty interior in the Raikov completion $\bar G$ of $G$.
\end{corollary}

\begin{corollary} Let $A$ be an analytic subgroup of a separable metrizable group $G$. If $A\notin\sigma\dot Z_n(X)$ for some $n\in\w$, then for any densely-Polish $n$-thick subsets $E,F\subset G$ the set $EAF^{-1}$ has non-empty interior in the completion $\bar G$ of $G$.
\end{corollary}

\begin{corollary} If for some $n\le\w$ a non-complete metrizable topological group $G$ contains a densely-Polish $n$-thick subset, then each analytic subset of $X$ belongs to the $\sigma$-ideal $\sigma \dot Z_n(X)\subset\sigma Z_n(X)$.
\end{corollary}

\begin{corollary} Let $A$ be an additive subset of an abelian metrizable topological group $G$. If $A\notin\sigma\dot Z_n(X)$ for some $n\le\w$, then for any densely-Polish $n$-thick subsets $E,F\subset X$ the set $A+E+F$ has non-empty interior in the completion $\bar G$ of $G$.
\end{corollary}

\begin{corollary} Let $A$ be an convex analytic subset of a linear metric space $X$. If $A\notin\sigma\dot Z_n(X)$ for some $n\le\w$, then for any densely-Polish $n$-thick subsets $E,F\subset X$ the set $A+E+F$ has non-empty interior in the completion $\bar X$ of $X$.
\end{corollary}

In light of the above results, it is important to recognize $n$-thick sets in topological groups and linear metric spaces. A characterization of $n$-thick convex sets is quite simple.

\begin{proposition}\label{p:convthick} For a convex subset $C$ in a separable linear metric space $X$ the following conditions are equivalent:
\begin{enumerate}
\item $C$ is $n$-thick in $X$ for every $n\le\w$;
\item $C$ is $n$-thick in $X$ for some $n\ge 1$;
\item the linear space $\IR\cdot(C-C)$ is dense in $X$;
\item the affine hull of $C$ is dense in $X$;
\item $C$ is $\{L\}$-thick in $X$ for some dense linear subspace $L$ of $X$.
\end{enumerate}
\end{proposition}

\begin{proof} We shall prove the implications $(1)\Ra(2)\Ra(3)\Ra(4)\Ra(5)$. The first implications $(1)\Ra(2)$ is trivial.

$(2)\Ra(3)$ Assuming that the convex set $C$ is $n$-thick in $X$ for some $n\ge1$, we shall prove that the linear space $L=\IR\cdot(C-C)$ is dense in $X$. Since $C$ is $n$-thick in $X$, there is a countable set $S\subset X$ such that the set $S+C$ is $n$-dense in $X$. Then the set $S+\bar L$ also is $n$-dense in $X$. Consider the quotient space $X/\bar L$ and the quotient linear operator $q:X\to X/\bar L$. Since the set $q(S+\bar L)=q(S)$ is countable, for each connected subspace $A$ of $S+\bar L$ the image $q(A)$ is a singleton, which means that contained in a single coset $x+\bar L$. Now the density of the $C(\II^n,S+\bar L)$ in $C(\II^n,\bar L)$ implies that $\bar L=X$.
\smallskip

$(3)\Ra(4)$ Assume that the linear space $L=\IR\cdot(C-C)$ is dense in $X$. Since for any point $c\in C$ the shift $c+L$ coincides with the affine hull $\aff(C)$ of $C$, the set $\aff(C)$ is  dense in $X$, too.
\smallskip

$(4)\Ra(5)$ Assume that the affine hull $\aff(C)$ of $C$ is dense in $X$. Replacing $C$ by a suitable shift, we can assume that zero belongs to $C$ and hence the affine hull of $C$ coincides with the linear hull of $C$. We shall prove that the convex set $C$ is $\{L\}$-thick for any dense linear subspace $L\subset \IR\cdot(C-C)$ of countable algebraic dimension. In this case we can find a countable subset $\{x_k\}_{k\in\w}$ in $C$ such that $x_0=0$ and the linear hull of the set $\{x_n\}_{n\in\w}$ contains the linear space $L$. For every $n\in\w$ by $\Delta_n$ and $L_n$ denote the convex and liner hulls of the finite set $F_n=\{x_0,\dots,x_n\}\subset C$. It is clear $L\subset \bigcup_{n\in\w}L_n$ and  $L_n=S_n+\Delta_n\subset S_n+C$ for some countable set $S_n\subset L_n$. Given a non-empty open subset $U\subset C$, we should find a countable set $S\subset X$ such that $L\subset S+U$. Fix any point $u\in U$ and find a neighborhood $\tilde U\subset X$ of zero such that $(u+\tilde U)\cap C\subset U$. For every $n\in\IN$ find a neighborhood $\tilde V\subset X$ of zero such that for any points $v_1,\dots,v_n\in \tilde V$ and real numbers $t_1,\dots,t_n\in[0,1]$ we get $\sum_{i=1}^nt_ix_i\in\tilde U$. Next, find $\e_n\in(0,1]$ such that $\e_n\cdot (F_n-u)\subset \tilde V$. The choice of $\tilde V$ guarantees that $\e_n(\Delta_n-u)\subset \tilde U$ and hence $$
\begin{aligned}
L_n&=(1-\e_n)u+\e_n\cdot L_n=(1-\e_n)u+\e_n (S_n+\Delta_n)=\e_nS_n+(1-\e_n)u+\e_n\Delta_n=\\
&=\e_nS_n+u+\e_n(\Delta_n-u)\subset \e_n S_n+(C\cap (u+\tilde U))\subset \e_n S_n+U.
\end{aligned}
$$ Then the countable set $S=\bigcup_{n=1}^\infty \e_nS_n$ has the required property: $L\subset \bigcup_{n=1}^\infty L_n\subset S+U$.
\smallskip

$(5)\Ra(1)$ Assume that $C$ is $\{L\}$-thick for some dense linear subspace $L\subset X$. By Lemma~\ref{dense}, $L$ is $\w$-dense in $X$, so $C$ is $\w$-thick and hence $n$-thick for every $n\le\w$.
\end{proof}

\begin{lemma}\label{dense} Let $A\subset B$ be convex sets in a linear metric space $X$. If $A$ is dense in $B$, then $A$ is $\w$-dense in $B$.
\end{lemma}

\begin{proof} It suffices to check that $A$ is $n$-dense in $B$ for every $n\in\IN$ (see \cite[V.2.1]{BP}). Given a continuous map $f:\II^n\to B$ and a neighborhood $U_0\subset X$ of zero, we need to find a continuous map $g:\II^n\to A$ such that $g(z)\in f(z)+U_0$ for all $z\in \II^n$. Choose an open neighborhood $W\subset X$ of zero such that for any points $w_0,\dots,w_n\in W+W-W$ and numbers $\lambda_0,\dots,\lambda_n\in \II=[0,1]$ we get $\sum_{i=0}^n\lambda_iw_i\in U_0$. Consider the open cover $\mathcal W=\{f^{-1}(x+W):x\in X\}$ of $\II^n$. Since $\II^n$ is an $n$-dimensional (para)compact space, there exists an finite open cover $\V$ of $\II^n$ such that for every $z\in \II^n$ the family $\V_z=\{V\in\V:z\in \V\}$ contains at most $n+1$ sets and its union  $\bigcup\V_z$ is contained in some set of the cover $\W$. By the paracompactness of $\II^n$, there is a partition of unity  $\{\lambda_V:\II^n\to[0,1]\}_{V\in\V}$ subordinated to the cover $\V$. The latter means that $\lambda_V^{-1}\big((0,1]\big)\subset V$ for all $V\in\V$, and $\sum_{V\in\V}\lambda_V\equiv1$. For every set $V\in\V$ fix a point $z_V\in V$ and by the density of $A$ in $B$ find a point $y_V\in A\cap (f(z_V)+W)$. Consider the map $g:\II^n\to L$ defined by the formula
$g(z)=\sum_{V\in\V}\lambda_V(z)y_V$ for $z\in\II^n$. It is clear that $g(\II_n)$ is contained in the convex hull $\Delta$ of the finite set $\{y_V\}_{V\in\V}\subset A$. We claim that $g(z)-f(z)\in U_0$ for all $z\in Z$. By the choice of the cover $\V$, the set $\bigcup\V_z$ is contained in some set $f^{-1}(W+x)$, $x\in X$. Then for every $V\in\V_z$ we get $f(z_V)-f(z)\in W-W$ and hence $y_V-f(z)\in W+f(z_V)-f(z)\subset W+W-W$. Then $g(z)-f(z)=\sum_{V\in\V_z}\lambda_V(z)(y_V-f(z))\in U_0$ by the choice of the neighborhood $W$. The map $g$ witnesses that $A$ is $n$-dense in $B$.
\end{proof}

A convex subset $C$ of a linear topological space $X$ is called {\em $\aff$-dense} in $X$ if the affine hull of $C$ is dense in $X$. By Proposition~\ref{p:convthick}, a convex subset of a separable linear metric space is $\aff$-dense if and only if it is $\w$-thick in $X$. 

\begin{theorem} If a non-complete linear metric space $X$ contains a densely-Polish $\aff$-dense convex set $C$, then every analytic subset of $X$ belongs to the $\sigma$-ideal $\dot \Z_{\{L\}}(X)$ for some dense linear subspace $L$ of $X$.
\end{theorem}

\begin{proof} Being densely-Polish, the convex set $C$ is separable and so is its affine hull $\aff(C)$. Since $\aff(C)$ is dense in $X$, the space $X$ is separable and its completion $\bar X$ is a Polish linear metric space. By Proposition~\ref{p:convthick}, the Polish convex set $C\subset X$ is $\{L\}$-thick for some dense linear subspace $L\subset X$. To finish the proof apply Corollary~\ref{c:non-complete-D} to the family $\DD=\{L\}$.
\end{proof}

For a separable linear metric space $X$ by $\mathcal L_\infty(X)$ we denote the family of dense linear subspaces in $X$. To simplify notation, denote the union $\bigcup_{L\in\LL_\infty(X)}\sigma \dot Z_{\{L\}}(X)$ by $\sigma \dot\Z_\infty(X)$. Observe that a set $A\subset X$ belongs to the family $\sigma\dot \Z_\infty(X)$ if and only if there exists a dense linear subspace $L\subset X$ (of countable algebraic dimension) in $X$ and a sequence $(A_n)_{n\in\w}$ of closed subsets of $X$ such that $A\subset \bigcup_{n\in\w}A_n$ and for every compact subset $K\subset L$ the sets $K+\bar A_n$, $n\in\w$, are nowhere dense in $X$.

It follows that $$\sigma \dot Z_\infty(X)\subset\sigma \dot \Z_\w(X)\subset\sigma\Z_\w(X)$$ for every separable linear metric space $X$.


\begin{theorem}\label{t:main-convex} For any analytic subsets $A,B\notin\sigma\dot\Z_\infty(X)$ of a linear metric space $X$ and any densely-Polish $\aff$-dense convex set $C$ in $X$ the sumset $A+B+C$ has non-empty interior in the completion $\bar X$ of $X$. Moreover, if $A$ is additive or convex, then the sum $A+C$ has non-empty interior in $\bar X$.
\end{theorem}

\begin{proof} By Proposition~\ref{p:convthick}, the $\aff$-dense convex sets $C$ is $\{L\}$-thick for some dense linear subspace $L$ of $X$. Then its homothetic copy $\frac12 C$ also is $\{L\}$-thick. The convexity of $C$ implies that $\frac12C+\frac12C\subset C$. Applying Corollary~\ref{c:main} to the family $\DD=\{L\}$ and observing that the $\sigma$-ideal $\sigma\dot\Z_{\{L\}}(G)\subset\sigma\dot\Z_\infty(G)$ does not contain the analytic sets $A,B$, we conclude that the sets $A+\frac12 C+B+\frac12C\subset A+B+C$ have non-empty interior in the completion $\bar X$ of $X$.
By the same reason, the sets $A+A+\frac12C+\frac12C\subset A+A+C$ and $A+A+C+C$ have non-empty interior in $\bar X$.

If $A$ is additive, then $A+A\subset A$ and hence the set $A+C\supset A+A+C$ has non-empty interior in $\bar X$. If $A$ is convex in $X$, then $\frac12(A+A)\subset A$ and hence the set $A+C\supset \frac12(A+A+C+C)$ has non-empty interior in $\bar X$.
\end{proof}

The following two theorems detect analytic groups and analytic convex sets which are $\sigma Z_\w$-spaces, thus giving partial positive answers to Problems~\ref{pr1} and \ref{pr2}.

\begin{theorem} An analytic subgroup $A$ of a linear metric space $X$ is a $\sigma Z_\w$-space provided that $A$ is not Polish and $A$ contains a densely-Polish $\aff$-dense convex subset $C$ of $X$.
\end{theorem}

\begin{proof} Since $A$ is a group, the set $\IN\cdot(C-C)$ is contained in the group $A$. The convexity of $C$ implies that $L=\IN\cdot (C-C)=\IR\cdot(C-C)$ is a linear subspace in $X$. The $\aff$-density of $C$ implies that the linear space $L\subset A$ is dense in $X$. By Lemma~\ref{dense}, the dense linear subspace $L$ is $\w$-dense in $X$ and so is the subgroup $A\supset L$. Since the sum $A+C=A$ has empty interior in $\bar X$, the set $A$ belongs to the $\sigma$-ideal $\sigma\dot \Z_\infty(X)\subset\sigma \Z_\w(X)$ by Theorem~\ref{t:main-convex}. Since $A$ is $\w$-dense in $X$, the inclusion $A\in\sigma\Z_\w(X)$ implies $A\in\sigma\Z_\w(A)$, which means that $A$ is a $\sigma Z_\w$-space.
\end{proof}

A similar result holds for convex sets.

\begin{theorem} A dense convex subset $A$ of a linear metric space $X$ is a $\sigma Z_\w$-space provided that $A$ is analytic, $A$ contains an $\aff$-dense densely-Polish convex subset $C$ of $X$ and $A$ has empty interior in the completion $\bar X$ of $X$.
\end{theorem}

\begin{proof} Since the sets $\frac12(A+C)\subset A$ has empty interior in $\bar X$, we can apply Corollary~\ref{c6} and conclude that $A\in\sigma\dot\Z_\infty(X)\subset\sigma\Z_\w(X)$. By Lemma~\ref{dense}, the dense convex subset $A$ of $X$ is $\w$-dense in $X$, which implies that $A\in\sigma\Z_\w(A)$.
\end{proof}

Finally, we study properties of analytic linear metric spaces containing  $\aff$-dense Polish convex sets.

A linear subspace $L$ of a linear metric space $X$ is called an {\em operator image} if $L=T(B)$ for some linear continuous operator $T:B\to X$ defined on a Banach space $B$. The topology of operator images was studied in \cite{BDP}. We shall prove that each $\aff$-dense Polish convex set in a linear metric space is $\{L\}$-thick for some dense operator image
$L\subset X$. For this we need the following known folklore fact.

\begin{proposition}\label{p:conv-sym} Each Polish convex set $A$ in a linear metric space contains a shift of a compact convex subset $K=-K$ such that the linear space $L=\IR\cdot K$ is dense in the linear hull of $A-A$.
\end{proposition}

\begin{proof} Replacing the convex set $A$ by a suitable shift of $A$, we can assume that $A$ contains zero.

Fix an invariant metric $d$ generating the topology of the linear metric space $X$ and let $\bar X$ be the completion of the linear metric space $(X,d)$. For a point $x\in X$ and a real number $\e>0$ by $B(x,\e)=\{y\in X:d(x,y)<\e\}$ and $\bar B(x,\e)=\{y\in X:d(y,x)\le\e\}$ we denote the open and closed $\e$-balls centered at $x$, respectively. The space $A$, being Polish, is a $G_\delta$-set in $\bar X$. So, we can write it as $A=\bigcap_{n\in\w}U_n$ for a descreasing family $(U_n)_{n\in\w}$ of open sets in $\bar X$. Fix a countable dense set $\{a_n\}_{n\in\w}$ in $A$.

Construct inductively two sequences of positive real numbers $(\e_n)_{n\in\w}$ and $(\lambda_n)_{n\in\w}$ such that for every $n\in\w$ the following conditions are satisfied:
\begin{enumerate}
\item $\max\{\lambda_n,\e_n\}<\frac1{2^{n+2}}$;
\item for every point $x$ in the compact set $\Delta_n=\{\sum_{k=0}^n t_k\lambda_ka_k:t_0,\dots,t_n\in[0,2]\}
    \}\subset A$ we get $\bar B(x,\e_n)\subset U_n$ and $x+[0,2\lambda_n]a_n\subset B(x,\e_n)$.
\end{enumerate}
The conditions (1), (2) imply that for every sequence $(t_n)_{n\in\w}\in[0,2]^\w$ the series $\sum_{n\in\w}t_n\lambda_na_n$ converges in $\bar X$ to some point  of the convex set $A=\bigcap_{n\in\w}U_n$. Put $c=\sum_{n\in\w}\lambda_na_n$ and observe that for every sequence $(t_n)_{n\in\w}\in[-1,1]^\w$ the series
$c+\sum_{n\in\w}t_n\lambda_na_n=\sum_{n\in\w}(1+t_n)\lambda_na_n$ converges to a point of $A$. It follows that the set $K=\{\sum_{n\in\w}t_n\lambda_na_n:(t_n)_{n\in\w}\in[-1,1]^\w\}$ is compact, convex, symmetric, and $c+K \subset A$.
It is clear that $\IR\cdot K\supset\{a_n\}_{n\in\w}$ is dense in the linear hull $\IR\cdot(A-A)$ of the set $A-A$.
\end{proof}

\begin{lemma}\label{l:op} If a linear metric space $X$ contains an $\aff$-dense Polish convex set $P$, then $X$ contains an $\aff$-dense compact convex set $K=-K$, which is $\{L\}$-thick for some dense operator image $L\subset X$.
\end{lemma}

\begin{proof} By Proposition~\ref{p:conv-sym}, there is a compact convex set $S=-S$ in $X$ such that $p+S\subset P$ for some $p\in P$ and the linear space $\IR\cdot S$ is dense in $\IR\cdot (P-P)$ and hence is dense in $X$. Choose a countable dense set $\{x_n\}_{n\in\w}$ in $S$ and find a sequence of real numbers $(\lambda_n)_{n\in\w}\in(0,1]^\w$ such that the linear operator $T:\ell_1\to X$, $T:(t_n)_{n\in\w}\mapsto\sum_{n=1}^\infty t_n\lambda_nx_n$, is well-defined and continuous. Here $\ell_1$ is the Banach space of real sequences $t=(t_n)_{n\in\w}$ with the norm $\|t\|=\sum_{n\in\w}|t_n|<\infty$. It is clear that the operator image $T(\ell_1)$ is dense in $X$. Denote by $B=\{t\in\ell_1:\|t\|\le 1\}$ the closed unit ball of the Banach space $\ell_1$ and let $K$ be the closure of the set $T(B)$ in $S$. It is clear that $K$ is a compact convex symmetric subset of $S$ and the affine hull $\IR\cdot K\supset T(\ell_1)$ is dense in $X$. We claim that the convex set $K$ is $\{T(\ell_1)\}$-thick. Given a non-empty open set $U\subset K$ we need to find a countable set $A\subset X$ such that $T(\ell_1)\subset A+U$. Since the set $T(B)$ is dense in $K$, the intersection $U\cap T(B)$ is not empty and hence the preimage $V=T^{-1}(U)$ contains some non-empty open subset of the ball $B$. The separability of the Banach space $\ell_1$ yields a countable set $A_1\subset\ell_1$ such that $\ell_1=A_1+V$. Then the countable set $A=T(A_1)$ has the required property:
$T(\ell_1)=T(A_1)+T(V)\subset A+U$.
\end{proof}

For a linear metric space $X$ denote by $\vec\LL_\infty(X)$ the family of dense operator images in $X$. To simplify notations, denote the family $\bigcup_{L\in\vec\LL_\infty(X)}\sigma\dot \Z_{\{L\}}(X)$ by $\sigma\vec\Z_\infty(X)$. Since $\vec\LL_\infty(X)\subset \LL_\infty(X)$, we get the inclusions
$$\sigma\vec\Z_\infty(X)\subset \sigma\dot \Z_\infty(X)\subset\sigma\dot\Z_\w(X)\subset\sigma\Z_\w(X).$$

\begin{proposition} A subset $A$ of a separable metric linear space $X$ belongs to the family $\sigma\vec\Z_\infty(X)$ if and only if there exists a $\sigma$-compact dense operator image $L$ in $X$ and a sequence $(A_n)_{n\in\w}$ of closed subsets of $X$ such that $A\subset\bigcup_{n\in\w}A_n$ for every  $n\in\w$ and compact subset $K\subset L$ the set $K\cdot A_n$ is nowhere dense in $X$.
\end{proposition}

\begin{proof} The ``if'' part of this proposition can be proved by analogy with Proposition~\ref{p:dotZn}. To prove the ``only'' if part, assume that $A\in \sigma\vec Z_\infty(X)$. Then $A\in\sigma\dot\Z_{\{L\}}(X)$ for some dense operator image $L$ in $X$. Write $L=T(B)$ for some linear continuous operator $T:B\to X$ defined on a Banach space $B$. Since the space $L$ is separable, we can find a separable Banach subspace $B'\subset B$ such that the operator image $L'=T(B')$ is dense in $T(B)$.  Choose a bounded sequence $(x_n)_{n\in\w}$ in $B'$ whose linear hull is dense in $B'$. It is standard to show that the operator $T':\ell_2\to B'$, $T':(t_n)_{n\in\w}\mapsto\sum_{n\in\w}\frac{t_n}{2^n}x_n$, is well-defined, compact, and has  dense image $T'(\ell_2)$ in $B'$. Then the operator $T'\circ T:\ell_2\to X$ is compact and has dense image $L'=T'\circ T(\ell_2)$ in $X$. It follows from $L'\subset L$ that $A\in \sigma\dot\Z_{L}(X)\subset\sigma\dot \Z_{L'}(X)$. So, we lose no generality assuming that $B=\ell_2$ and the operator $T$ is compact. By the compactness of the operator $T$ and the reflexivity of $\ell_2$, the image $T(B_1)$ of the closed unit ball in $B_1$ of the Hilbert space $\ell_2$ is compact. This implies that the operator image $L=T(\ell_2)$ is $\sigma$-compact.
Since $A\in\sigma\dot\Z_{\{L\}}(X)$, there is a sequence $(A_n)_{n\in\w}$ of closed subsets $A_n$ of $X$ such that $L+\bar A_n$ has empty interior in $X$. Then for every compact subset $K\subset L$ the closed set $K+\bar A_n$ has empty interior and hence is nowhere dense in $X$. Then the set $K+A_n$ is nowhere dense in $X$.
\end{proof}

Applying Corollary~\ref{c:main} and Lemma~\ref{l:op} to the family $\vec\LL_\infty(X)$, we can prove the following corollary (by analogy with Theorem~\ref{t:main-convex}).

\begin{theorem} For any analytic subsets $A,B\notin\sigma\vec\Z_\infty(X)$ of a linear metric space $X$ and any  $\aff$-dense convex Polish set $C$ in $X$, the sumset $A+B+C$ has non-empty interior in the completion $\bar X$ of $X$. Moreover, if $A$ is additive or convex, then the sumset $A+C$ has non-empty interior in $\bar X$.
\end{theorem}

This theorem has

\begin{corollary} If a non-complete linear metric space $X$ contains a Polish $\aff$-dense convex set $C$, then every analytic subset of $X$ belongs to the $\sigma$-ideal $\dot \Z_{\{L\}}(X)$ for some dense operator image $L$ in $X$.
\end{corollary}

\end{document}